\documentclass[letterpaper, 10 pt, conference]{ieeeconf}  % Comment this line out if you need a4paper

\IEEEoverridecommandlockouts                              % This command is only needed if 
\usepackage{xifthen}
\usepackage{xcolor}
\usepackage{hyperref}
\definecolor{dark}{rgb}{0.0,0.0,0.0}
\hypersetup{colorlinks,breaklinks,
            linkcolor=dark,urlcolor=dark,
            anchorcolor=dark,citecolor=dark}

\usepackage{amsmath}
\usepackage{amssymb}
\usepackage{eucal}
\usepackage{graphicx}

\usepackage{algorithm}
\usepackage{lscape}
\usepackage{subfig}
\usepackage{tikz}
\usetikzlibrary{calc}
\usetikzlibrary{decorations.markings}
\usetikzlibrary{matrix}

\usepackage{cite}
\usepackage{amsmath,amsfonts}
\usepackage{fixltx2e}
\usepackage{array}
\usepackage{multirow}
\usepackage{booktabs}

\usepackage{comment}
\usepackage{microtype}

\newcommand{\real}[1]{\mathbb{R}^{#1}{}}

\newcommand{\bmat}[1]{\begin{bmatrix}#1\end{bmatrix}}

\newcommand{\transpose}{^T}

\newcommand{\inverse}{^{-1}}

\newcommand{\df}{\dot{f}}
\newcommand{\ddf}{\ddot{f}}

\newcommand{\myparagraph}[1]{\textbf{\emph{#1}}.}

\newcommand{\Sset}{\mathcal{S}}

\newcommand{\ie}{\textsl{i.e.}{}}
\newcommand{\eg}{\textsl{e.g.}{}}

\newtheorem{theorem}{Theorem}[section]

\newtheorem{proposition}[theorem]{Proposition}
\newtheorem{lemma}[theorem]{Lemma}
\newtheorem{problem}[theorem]{Problem}
\newtheorem{assumption}[theorem]{Assumption}

\DeclareMathOperator*{\trace}{tr}

\DeclareMathOperator{\sign}{sign}

\newcommand{\D}{\mathrm{D}}
\DeclareMathOperator{\grad}{grad}

\newcommand{\gradx}{\grad_x\!}
\newcommand{\Dxgradx}{\D_x\!\grad_x\!}

\newcommand{\abs}[1]{\lvert#1\rvert}

\newcommand{\norm}[1]{\lVert#1\rVert}

\newcommand{\defeq}{\doteq}
\newcommand{\de}{\mathrm{d}}
\newcommand{\dert}[1][]{\frac{\de #1}{\de t}}

\newcommand{\dtheta}{\dot{\theta}}

\newcommand{\tx}{\tilde{x}}
\newcommand{\dtx}{\dot{\tx}}
\newcommand{\betagi}{\beta_{gi}}
\newcommand{\tbetagi}{\tilde{\beta}_{gi}}
\newcommand{\dtbetagi}{\dot{\tilde{\beta}}_{gi}}
\newcommand{\betai}{\beta_i}

\newcommand{\tbetai}{\tilde{\beta}_i}
\newcommand{\dtbetai}{\dot{\tilde{\beta}}_i}
\newcommand{\ri}{d_i}
\newcommand{\rgi}{d_{gi}}
\newcommand{\rgmax}{d_{g,\textrm{max}}}
\newcommand{\tri}{\tilde{d}_i}
\newcommand{\dtri}{\dot{\tilde{d}}_i}
\newcommand{\ci}{c_i}
\newcommand{\ciinf}{c_{i,\infty}}
\newcommand{\tci}{\tilde{c}_i}
\newcommand{\dtci}{\dot{\tilde{c}}_i}
\newcommand{\varphii}{\varphi_i}
\newcommand{\tvarphi}{\tilde{\varphi}}
\newcommand{\tvarphii}{\tvarphi_i}
\newcommand{\dtvarphi}{\dot{\tilde{\varphi}}}
\newcommand{\dtvarphii}{\dtvarphi_i}
\newcommand{\numax}{\nu_{\textrm{max}}}
\newcommand{\omegamax}{\omega_{\textrm{max}}}
\newcommand{\Dbetagradx}{\D_{\betagi}\!\grad_x\!}
\definecolor{myorange}{RGB}{255,168,19}
\definecolor{mygreen}{RGB}{157,178,37}
\definecolor{myred}{RGB}{204,55,42}
\definecolor{myblue}{RGB}{3,119,201}

\tikzset{
  dim below/.style={to path={\pgfextra{
        \pgfinterruptpath
        \draw[>=latex,|->|] let 
        \p1=($(\tikztostart)!1em!-90:(\tikztotarget)$),
        \p2=($(\tikztotarget)!1em!90:(\tikztostart)$)
        in (\p1) -- (\p2) node[pos=.5,sloped,below]{#1};
        \endpgfinterruptpath
      }
    }
  }
}
\catcode`~=11 % make LaTeX treat tilde (~) like a normal character
\newcommand{\urltilde}{\kern -.15em\lower .7ex\hbox{\texttt{~}}\kern .04em}
\catcode`~=13 % revert back to treating tilde (~) as an active character

\newboolean{isconference}

\setboolean{isconference}{false}

\title{\LARGE \bf
  \ifthenelse{\boolean{isconference}}{
    An Optimization Approach to Bearing-Only Visual Homing\\with Applications to a 2-D Unicycle Model
  }{
    Technical report on Optimization-Based Bearing-Only Visual Homing with Applications to a 2-D Unicycle Model
  }
}

\author{Roberto Tron and Kostas Daniilidis% <-this % stops a space
\thanks{The authors are with the Department of
Computer and Information Science, University of Pennsylvania, Philadelphia, PA 19104,
{\tt\small\{tron,kostas\}@seas.upenn.edu}. This work was supported by grants NSF-IIP-0742304, NSF-OIA-1028009, ARL MAST CTA W911NF-08-2-0004, ARL Robotics CTA W911NF-10-2-0016, NSF-DGE-0966142, and NSF-IIS-1317788.}%
}

\begin{document}

\maketitle
\thispagestyle{empty}
\pagestyle{empty}

%%%%%%%%%%%%%%%%%%%%%%%%%%%%%%%%%%%%%%%%%%%%%%%%%%%%%%%%%%%%%%%%%%%%%%%%%%%%%%%%
\begin{abstract}
We consider the problem of bearing-based visual homing: Given a mobile robot which can measure bearing directions with respect to known landmarks, the goal is to guide the robot toward a desired ``home'' location. We propose a control law based on the gradient field of a Lyapunov function, and give sufficient conditions for global convergence. We show that the well-known Average Landmark Vector method (for which no convergence proof was known) can be obtained as a particular case of our framework. We then derive a sliding mode control law for a unicycle model which follows this gradient field. Both controllers do not depend on range information. Finally, we also show how our framework can be used to characterize the sensitivity of a home location with respect to noise in the specified bearings. \ifthenelse{\boolean{isconference}}{}{This is an extended version of the conference paper \cite{Tron:ICRA14}.}
\end{abstract}

%%%%%%%%%%%%%%%%%%%%%%%%%%%%%%%%%%%%%%%%%%%%%%%%%%%%%%%%%%%%%%%%%%%%%%%%%%%%%%%%
\section{INTRODUCTION}
\label{sec:introduction}

Assume that a mobile robot is equipped with a camera, that a picture of the environment is taken from a ``home'' location, and that the robot is subsequently moved to a new location, from which a new picture is taken. Assume that few landmarks can be detected in both images. The robot can then obtain the bearing (\ie, the direction vector) for each landmark with respect to the home and new locations. Bearing-based visual homing problem (Figure~\ref{fig:problem}) aims to find a control law that guides the robot toward the home location using only the extracted bearing information.

There is evidence that visual homing is naturally used as a mechanism to return to the nest by particular species of bees \cite{Cartwright:JCP83} and ants \cite{Wehner:JCPA03,McLeman:IS02}. % such as the desert ant \emph{Cataglyphis} \cite{Wehner:JCPA03} and the \emph{Leptothorax albipennis} \cite{McLeman:IS02}, 
In robotics, it has been used to navigate between nodes in a topological map of an environment \cite{Hong:CSM92}. %Bearing-based visual homing is an attractive approach because makes minimal assumptions on the environment.

\myparagraph{Previous work} In the last fifteen years, a number of approaches have been proposed in the literature. %Broadly, they can be classified depending on what knowledge of the environment they require (full location of the landmarks, global compass direction), the dimensionality in which they can be applied (2-D, 3-D), constraints on the motion of the robot and theoretical convergence guarantees (local, global).
The majority of the approaches assumes that the position of the robot can be controlled directly (\ie, there are no non-holonomic constraints) and that a global compass direction is available (so that bearing directions can be compared in the same frame of reference). Historically, the first approach proposed for robotics applications is the Average Landmark Vector (ALV) method \cite{Lambrinos:GNC98}, where the difference of the averages of the bearing vectors at the home and current locations is used as the control law. This has the computational advantage of not requiring explicit matching. % between bearing vectors.
 However, until now, its convergence has been proved only empirically. %The basic idea is to combine all the landmark vectors into an single average vector. The difference between the averages at the home and current locations is then used as the motion direction.
 %The main advantage of this strategy is that it does not require explicit matching between the bearings, and, therefore, it is the most computationally efficient.
% However, its convergence has been proved only empirically, and it appears to fail when the landmarks are concentrated in a small region of the field of view \cite{Angulo:AIRD07}.

\begin{figure}[t]
  \centering
  \begin{tikzpicture}[point/.style={circle,fill,inner sep=0pt,minimum size=1em},
    landmark/.style={point,myblue},
    home/.style={point,mygreen},
    current/.style={point,myred},
    descLabel/.style={draw,rounded corners,thick}
    ]
    %points
    \coordinate[landmark] (l1) at (-2,0.4);
    \coordinate[landmark] (l2) at (-2.5,-1);
    \coordinate[landmark] (l3) at (3,-0.5);
    \coordinate[landmark] (l4) at (-2.8,-2);
    \coordinate[home] (h) at (-0.2,-0.1);
    \coordinate[current] (c) at (0.5,-1.7);

    %lines and vectors
    \foreach \base/\prop in {h/mygreen,c/myred} {
      \foreach \i in {1,...,4} {
        \draw[\prop,-latex,thick] (\base) -- ($(\base)!1cm!(l\i)$) coordinate (\base-l\i);
        \draw[\prop,dashed] (\base) -- (l\i);
      }
    }
    \draw[-latex,dashed,thick,black] (c) -- (h);

    %notation labels
    \node[anchor=south,inner sep=0.75em] at (h) {$x_g$};
    \draw (c) +(80:0.6em) node[anchor=south west] {$x$};
    \node[anchor=south,inner sep=0.75em] at (l3) {$x_i$};
    \node[anchor=north] at (h-l3) {$\betagi$};
    \node[anchor=south] at (c-l3) {$\betai$};
    \draw (c) to[dim below=$\ri$] (l3);

    %description labels
    \draw (h) +(20:1.5em) node[anchor=south west,mygreen,descLabel] {Home location};
    \draw (c) +(-90:1em) node[anchor=north east,myred,descLabel] {Current location};
    \draw (l1) +(-150:2.2em) node[anchor=north,myblue,descLabel] {Landmarks};
  \end{tikzpicture}
  \caption{Illustration of the visual homing problem and notation.}
\label{fig:problem}
\end{figure}
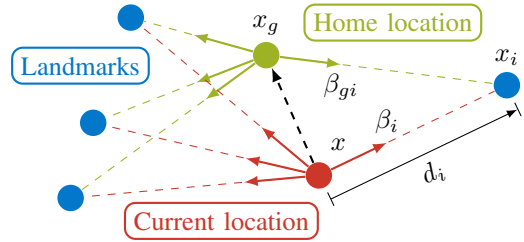
 
Another approach is the snapshot-model, which originated as a computational model for bees \cite{Cartwright:JCP83} and was first implemented in robotics in \cite{Moller:CVMRW98}. This model matches the bearings between the two locations, and uses their differences (possibly weighted with the apparent size of the landmarks) to compute the control vector. A comparison of some of these approaches under noisy conditions is given in \cite{Schroeter:IAS08}. Again, there is no theoretical proof of convergence, and it appears that this strategy works well only when the landmarks are well distributed \cite{Goedeme:CIRA05}. Variations of this approach \cite{Argyros:CVPR01,Lim:RSS09,Liu:ICRA13} use differences between the bearing angles instead of the direction themselves, which removes the assumption of a compass direction. While \cite{Argyros:CVPR01} combines all contributions together, \cite{Lim:RSS09,Liu:ICRA13} are based on small groups of bearings. %(triplets for the 2-D case, or quadruplets for the 3-D case) to compute the control direction. %If more bearings are available, multiple groups can be combined in a voting scheme \cite{Lim:RSS09} or used in turns \cite{Liu:ICRA13}. %Thanks to the use of differences, these approaches do not require a global compass direction. 
No proof of convergence is given in \cite{Argyros:CVPR01}, while \cite{Lim:RSS09,Liu:ICRA13} show that the computed control direction always brings the robot closer to the home position, although no global convergence is shown. Moreover, the authors do not provide a way to decide the \emph{scale} of the motion.
 In contrast, \cite{Loizou:CDC07} requires a global compass direction and is limited to the 2-D case, but it provides global convergence proofs. %it proposes a control law involving all the bearings at the same time and shows global convergence.% using a Lyapunov function.

A different line of work considers the problem as an instance of image-based visual servoing (IBVS) \cite{Corke:RR03}. The main limitation of this approach is that it requires range information. % in addition to the bearings.
In practice, this information can be approximated with constant ranges, estimated by creating a map of the landmarks on-line \cite{Goedeme:CIRA05}, through the use of adaptive control \cite{Papanikolopoulos:TAC93} or using the relative scale of feature point descriptors \cite{Liu:ICRA12}. However, even with range information, only local convergence can be shown. A notable exception in this sense is Navigation Function (NF) approach given by \cite{Cowan:TRA02,Koditschek:book89}, which allows the specification of safety areas (\eg, such that the target is always visible) and the use of Lagrangian (second-order) dynamical models. However, it requires planar targets with known geometry, hence it does not apply to our situation.

Yet another line of work is based on the use of the \emph{essential matrix} \cite{Basri:ICCV98,Piazzi:CPR03,Rives:IROS00} in two-view geometry, which can be directly computed from the correspondences and provides an estimate of the relative pose between the home and current location. This approach does not require a global compass direction, but the estimation of the essential matrix becomes ill-conditioned in degenerate configurations (\eg, near the target home position), %, or when the landmarks are far way, or when the scene is planar.
which need to be detected and addressed separately. 

Regarding the use of non-holonomic mobile robots, the work of \cite{Usher:ICRA03} extends the IBVS approach to a car-like model, while \cite{Wei:IJISTA05} extends the ALV approach and \cite{Mariottini:ICRA04,Mariottini:TRA07,LopezNicolas:ICRA06,Becerra:IROS08} extend the essential matrix approach to a unicycle model. The same limitations as the original works apply. %In addition, for \cite{Mariottini:ICRA04,Mariottini:TRA07,LopezNicolas:ICRA06,Becerra:IROS08}, a singularity of the control law appears when the axis of the robot is aligned with the line between the home and current locations. Again, this condition needs to be detected and properly addressed. 
An improvement is given by \cite{Aranda:AR13}, which uses the 1-D \emph{trifocal tensor} (an extension of the essential matrix to three views) to compute the relative angles between the desired and current bearings for the control law. This approach requires correspondences between four views, and is limited to 2-D bearings.

\myparagraph{Paper contributions}
In this paper, we make the assumption that a global compass direction is available. In \S\ref{sec:control-an-holonomic}, we focus on holonomic mobile robots and propose a control law derived from the gradient of a cost belonging to carefully chosen class of functions. We can show that this gradient does not depend on range information, and we give sufficient conditions for global convergence to a unique minimizer. Interestingly, for a specific choice of cost, we recover the ALV model, thus providing a rigorous proof of its global convergence (which, to the best of our knowledge, had not appeared before).%our cost function is akin to the NF of \cite{Cowan:TRA02}. However, we work only with Newtonian (first-order) dynamical models and without any assumption on the landmark configuration.
 Our approach applies to both the 2-D and 3-D cases.
In \S\ref{sec:sliding-mode-unicycle}, we extend our approach to a unicycle model using sliding mode control. % As in \cite{Becerra:IROS08}, we use sliding mode control. 
We cast this application as a specific instance of the more general problem of moving a unicycle along the flow trajectories of an arbitrary gradient field, using only upper and lower bounds on the differential of the gradient.
Finally, in \S\ref{sec:sensitivity-analysis}, we show how our optimization framework can be used to generate heat maps indicating, for a given landmark configuration, the sensitivity of the homing position with respect to noise in the specification of the bearing. This kind of analysis has not appeared in any of the prior work.

%%% Local Variables: 
%%% mode: latex
%%% TeX-master: "ICRA13-BearingControl"
%%% End: 

\section{NOTATION AND DEFINITIONS}
We collect in this sections various notation conventions (see also Figure~\ref{fig:problem}) and other definitions.
We denote as $x_g\in\real{d}$ the coordinates of the home location and as $\{x_i\}_{i=1}^N\in \real{d}$ the location of the $N$ landmarks. For a given $i\in\{1,\ldots,N\}$, the bearing direction at a location $x\in\real{d}$ is defined as
\begin{equation}
  \betai(x)=\ri(x)\inverse (x_i-x),  
\end{equation}
where $\ri(x)$ contains the range information
\begin{equation}
  \ri(x)=\norm{x_i-x}.  
\end{equation}
We denote as $\betagi$ the home bearing directions, \ie,
\begin{equation}
  \label{eq:bearing-home}
  \betagi=\betai(x_g).
\end{equation}
We also define the following inner product:
\begin{equation}
  c_i(x)=\betagi\transpose \betai(x).
\end{equation}

In order to reduce clutter, when possible, we omit the explicit dependence on $x$ (\eg, we use $\betai$ instead of $\betai(x)$). 

We assume that the set of home bearing directions $\{\betagi\}$ is \emph{consistent}, \ie, that it exists $x_g\in\real{d}$ such that \eqref{eq:bearing-home} is satisfied for all $i\in\{1,\ldots,N\}$. Moreover, we assume that they are \emph{non-degenerate}, \ie, that such point $x_g$ is uniquely determined. It is easy to see that the bearing directions are non-degenerate if and only if we have $N\geq 2$ points non-collinear with $x_g$ (one bearing constraints the position of $x_g$ on a line, and another bearing fixes a point on this line).
%Intuitively, a set of home bearing directions is both consistent and non-degenerate if and only if they unambiguously identify a point in $\real{d}$.

% Given a function $\varphi:\real{d}\to\real{}$ and a value $\bar{\varphi}<\infty$, we define the level set $S_{\bar{\varphi}}$ as
% \begin{equation}
%   S_\varphi(\bar{\varphi})=\{x\in\real{d}:\varphi(x)\le\bar{\varphi}\}.
% \end{equation}

We consider mobile robots governed by a particular cases of affine, driftless dynamical system of the form
\begin{equation}
  \label{eq:dynamical-model}
  \dot{z}=h(z)u,
\end{equation}
where $z$ represents the state of the system (such as the robot's location) and $u$ represents the inputs available for control. In particular, we will consider the case of a simple integrator (\S\ref{sec:control-an-holonomic}) and of a 2-D unicycle model (\S\ref{sec:sliding-mode-unicycle}).

We formally state the visual homing problem as follows.
\begin{problem}
  Given a set of desired bearing directions $\{\betagi\}$, define a control law $u$ for \eqref{eq:dynamical-model}, using $\{\betai\}$ alone, which asymptotically drives the robot's location to $x_g$.
\end{problem}

In various places, $\tx(t)=x_0+tv$ denotes a parametrized line in $\real{d}$, where $x_0,v\in\real{d}$ are arbitrary. The notation $\tilde{\cdot}$ indicate a function evaluated along $\tx(t)$, \eg, $\tvarphi(t)=\varphi\bigl(\tx(t)\bigr)$. % In this case, the values of $x_0$ and $v$ will be clear from the context.
Also, we denote as $P_v=I-vv\transpose$ the projection matrix on the space orthogonal to the normalized vector $v$.

Finally, given a mapping $\Phi:\real{d_1}\to\real{d_2}$, its \emph{differential} or \emph{Jacobian} is defined as the matrix $\D_x\Phi:\real{d_2\times d_1}$ such that
\begin{equation}
  \label{eq:definition-differential}
  \D_x\Phi v=\left.\dert \Phi(\tx)\right|_{t=0}
\end{equation}
for any choice of $v$ in $\tx$. This definition coincides with the usual one as a matrix of derivatives. As it is standard, we define the \emph{gradient} of a function $\varphi:\real{d}\to\real{}$ as
\begin{equation}
  \label{eq:definition-gradient}
  \gradx\varphi=\D_x\varphi\transpose.  
\end{equation}
 
% With this notation, the gradient of a function $\varphi:\real{d}\to\real{}$ at $x_0$ can be defined as the vector $\grad\varphi\in\real{d}$ such that
% \begin{equation}
%   \label{eq:definition-grad}
%   v\transpose \grad\varphi(x_0)=\left.\dert \tvarphi\right|_{t=0}
% \end{equation}
% for any choice of $v$ in $\tx$.
% Similarly, the differential of the gradient can be defined as the matrix $\Dgrad\varphi\in\real{d\times d}$ such that
% \begin{equation}
%   \label{eq:definition-differential}
%   \Dgrad\varphi(x_0) v=\left.\dert \grad(\tx)\right|_{t=0}
% \end{equation}
% for any choice of $v$ in $\tx$.

\section{CONTROL OF AN HOLONOMIC ROBOT}
\label{sec:control-an-holonomic}
In this section, we will consider a linear integrator of the form $\dot{x}=u$, \ie, in \eqref{eq:dynamical-model}, we choose $z=x$,
where $x\in\real{d}$ represents the location of the robot and $h=I$, the identity matrix.
We define our control law from as a gradient descent of a cost function $\varphi:\real{d}\to\real{}$:
\begin{equation}
  \label{eq:control-input}
  u=-\gradx\varphi.
\end{equation}

To be more specific, given a univariate reshaping function $f:\real{}\to\real{}$, we define the cost function $\varphi$ as
\begin{align}
  %\textstyle 
  \label{eq:cost}
  \varphi=\sum_{i=1}^N \varphi_i,\quad \varphii=r_i f(c_i).
\end{align}
Note that $\varphi$ is not defined on $\{x_i\}$. As we will show later, $f$ can be chosen such that $\varphi$ is radially unbounded \cite{khalil:book02} and that it has a global minimum at $x_g$ and no other critical points. This will imply global convergence of the closed-loop system. For the moment, we have the following result
\begin{proposition}\label{prop:grad-and-Dgrad}
The gradient of \eqref{eq:cost} is given by
\begin{equation}
\begin{aligned}
  \label{eq:gradient-cost}
  \gradx \varphi &= \sum_{i=1}^N \gradx\varphii, \\ \gradx\varphii&=-f(c_i) \betai - \df(c_i) P_{\betai} \betagi,
\end{aligned}
\end{equation}
while its differential is given by
\begin{gather}
  \label{eq:Dgradient-cost}
  \Dxgradx\varphi\!=\!\sum_{i=1}^N \Dxgradx\varphii, \;\;\Dxgradx\varphii=\ri\inverse H_i,\\
  H_i= \bigl(\ddf(\ci) (\ci\betai-\betagi)\betagi\transpose+(\df(\ci)\ci-f(\ci))I\bigr)P_{\betai}.
\end{gather}
\end{proposition}
See \S\ref{sec:proof-gradient} in the Appendix for a proof. In this section we use only \eqref{eq:gradient-cost}, while \eqref{eq:Dgradient-cost} will be used in \S\ref{sec:sliding-mode-unicycle}.

Notice that, although the cost $\varphi$ depends on the range information $\{\ri\}$, the gradient $\gradx\varphi$, and hence the control law \eqref{eq:control-input}, depend only on the bearing directions $\{\betai\}$ and $\{\betagi\}$. Regarding the choice of $f$, we use the following:
\begin{assumption}
\label{ass:reshaping}
  The function $f:[-1,1]\to\real{}$ satisfies:
\begin{align}
  f(1)&=0,\label{eq:property-f-minval}\\
  \df(c)&\begin{cases}\leq0& \textrm{and finite for } c=1,\\<0 &\textrm{otherwise},\end{cases}\label{eq:property-f-firstderivative}\\
%  \df_b(c)&\textrm{ is finite}, \label{eq:property-f-finitederivative}\\
  f(c)+(1-c)\df(c)&\leq 0. \label{eq:property-f-derivativemix}
\end{align}
\end{assumption}

Note that \eqref{eq:property-f-minval} with \eqref{eq:property-f-firstderivative} implies that $f$ is non-negative,
and that induces an ordering of the cosines,
\begin{equation}
  \label{eq:ordering}
  f(c_1)>f(c_2)\textrm{ if } c_1<c_2.
\end{equation}

We first show that $\varphi$ can be used as a cost function for visual homing, \ie, that the function has a global minimizer at the correct location. 
\begin{lemma}\label{lemma:global-minimum}
  The function $\varphi$ is everywhere non-negative and has a unique global minimizer at $x_g$.
\end{lemma}
\begin{proof}
 ce $\ri$ and $f(c_i)$ are non-negative, each term in $\varphi$ is non-negative. It follows that $\varphi$ is also non-negative, and it is zero iff each term $\varphii$ is zero for all $i$. From \eqref{eq:property-f-minval}, this implies $c_i=1$ for all $i$, that is, the current bearings coincide with the desired ones, and $x=x_g$.
\end{proof}
As we are interested in global convergence, we need to show that there are no other local minimizers. At a high level, due to the fact that the function $\varphi$ is in general non-convex, we proceed by showing that the cost is always increasing when evaluated along lines emanating from $x_g$ in arbitrary directions. From this, we will exclude the presence of unwanted local minimizers. In order to carry out this program on the entire cost $\varphi$, we first prove a similar result on each of its terms $\varphii$.
\begin{lemma}\label{lemma:derivative-lines}
  Define the line $\tx(t)=x_g+tv$, where $v\in\real{d}$ is arbitrary. Then, the derivative of the function $\tvarphii(t)=\varphi_i\bigl(\tx(t)\bigr)$ satisfies the following. If $t=0$ or $v=a\betagi$, $a\leq 0$, then $\dtvarphii\equiv0$. Otherwise, $\dtvarphii>0$ for all $t>0$, except when $v=a\betagi$, $a>0$, for which
  \begin{equation}
    \dtvarphii
    \begin{cases}
      = 0 & \textrm{for } t\in \left[0,\frac{\norm{x_i-x_g}}{\norm{v}}\right),\\
      > 0 & \textrm{for } t>\frac{\norm{x_i-x_g}}{\norm{v}}.\\
    \end{cases}
  \end{equation}
\end{lemma}
The proof of this lemma is somewhat involved, and can be found in \S\ref{sec:proof-derivative-lines} of the Appendix.

We then use this result on the entire cost function $\varphi$ and obtain the following.

\begin{proposition}\label{prop:unique-critical-point}
  The function $\varphi$ has only a global minimizer and no other critical points.
\end{proposition}
\begin{proof}
  From Lemma~\ref{lemma:global-minimum}, $x_g$ is a global minimizer. Consider an arbitrary point $x_0\neq x_g$ and the curve $\tx(t)=x_g+t(x_0-xg)$ (notice that $\tx(1)=x_0$). If we can show that $\left.\dert\varphi(\tx)\right|_{t=1}\neq 0$, then, by definition~\eqref{eq:definition-differential}, $\gradx\varphi(x_0)\neq 0$, and $x_0$ cannot be a critical point. In fact, by linearity,
  \begin{equation}
    \label{eq:derivative-decomposition}
    \left.\dert\varphi(\tx)\right|_{t=1}=\sum_{i=1}^N \left.\dert\varphii(\tx)\right|_{t=1}.
  \end{equation}
From Lemma~\ref{lemma:derivative-lines}, each term on the RHS of \eqref{eq:derivative-decomposition} is non-negative, and, since the bearing directions are non-degenerate and non-collinear, at least one is strictly positive. Hence $\left.\dert\varphi(\tx)\right|_{t=1}>0$, $\gradx\phi(x_0)\neq 0$ (from \eqref{eq:definition-differential}--\eqref{eq:definition-gradient}), and $\varphi$ cannot have a critical point at $x_0$.
\end{proof}

The final ingredient we need for our convergence result is to show that the trajectories of the closed-loop system alway lie in a compact set. This is implied by the following.
\begin{proposition}\label{prop:radially-unbounded}
 The function $\varphi$ is radially unbounded.  
\end{proposition}
The proof relies on showing that at least one term $\varphii$ grows unbounded in any arbitrary direction%
\ifthenelse{\boolean{isconference}}{%
{} (see \cite{Tron:Arxiv14} for details).
}{%
, and it can be found in \S\ref{sec:proof-radially-unbounded} of the Appendix.%
}
We then arrive to the following theorem, which represents our first main result and shows that the control law \eqref{eq:control-input} provides global asymptotic convergence.
\begin{theorem}\label{thm:convergence-integrator}
All the trajectories of the closed-loop system
\begin{equation}
  \label{eq:feedback-system}
  \dot{x}=-\gradx{\varphi(x)}
\end{equation}
converge asymptotically to $x_g$.
\end{theorem}
The solutions to \eqref{eq:feedback-system} are also known as the \emph{gradient flow}.
This result follows easily from Propositions~\ref{prop:radially-unbounded} and~\ref{prop:unique-critical-point} and a standard Lyapunov stability argument.

In practice, we have different options for the function $f$. The following proposition lists a couple of them.
\begin{proposition}\label{prop:functions}
  These functions satisfy Assumption~\ref{ass:reshaping}:
  \begin{align}
    f(c)&=1-c \label{eq:f-cosine}\\
    f(c)&=\frac{1}{2}\arccos^2(c) \label{eq:f-angleSq}
%    f(c)&=\frac{1}{2}(c-1)^2 \label{eq:f-cosineSq}
  \end{align}
\end{proposition}
Notice that \eqref{eq:f-cosine} is the simplest function one can think of, and it is related to the cosine of the angle between the bearings, while \eqref{eq:f-angleSq} represents the square of the angle itself.

The proof of this proposition (a simple direct calculation), can be found in %
\ifthenelse{\boolean{isconference}}{%
\cite{Tron:Arxiv14}.
}{%
\S\ref{sec:proof-functions} of the Appendix.
}
Note that by choosing \eqref{eq:f-cosine}, the gradient \eqref{eq:gradient-cost} simplifies to
\begin{equation}
  \label{eq:gradient-cost-f-cosine}
  \gradx\varphi=-\sum_{i=1}^N \betai+\sum_{i=1}^N \betagi,
\end{equation}
which is equivalent to the ALV method mentioned in the introduction. In this case, we do not need to know the exact correspondences between $\{\betai\}$ and $\{\betagi\}$ (the two sums in \eqref{eq:gradient-cost-f-cosine} can be computed separately).

\section{SLIDING MODE CONTROL OF A UNICYCLE}
\label{sec:sliding-mode-unicycle}
In this section, we will extend the results above to a robot governed by a non-linear 2-D unicycle model of the form \eqref{eq:dynamical-model} where
\begin{align}
  \label{eq:unicycle-model}
  z=\bmat{x\\\theta}, && h=\bmat{\cos(\theta) & 0\\ \sin(\theta) & 0\\ 0 & 1}, && u=\bmat{\nu\\\omega},  
\end{align}
and where $x\in\real{2}$ represents the location of the robot, $\theta\in(-\pi,\pi]$ represents the orientation of the unicycle, and $\nu,\omega\in\real{}$ represents the linear and angular velocity inputs, respectively. Since we make the assumption that a global compass direction is available, we assume $\theta$ to be known.

Our main goal is to make the unicycle follow the flow of the gradient field of $\varphi$. More precisely, we will propose a control law such that, with the model given by \eqref{eq:dynamical-model} and \eqref{eq:unicycle-model},
\begin{equation}
  \label{eq:asymptotic-xdot}
  \dot{x}=-k_\nu\gradx\varphi \quad \forall t>T,  
\end{equation}
for some $T<\infty$.
The arbitrary constant $k_\nu>0$ fixes the ratio between the magnitude of the gradient and the speed of the unicycle.
Since we have shown in \S\ref{sec:control-an-holonomic} that the flow trajectories of the gradient field converge asymptotically toward $x_g$, then \eqref{eq:asymptotic-xdot} implies that $x$ converges toward $x_g$ also for the unicycle model. Note that, if necessary, a second control phase can then be applied to rotate the unicycle to a desired pose.
At first, this might appear to be an instance of a path-following problem, where we the cart should follow a prescribed trajectory, and for which solutions already exist. However, in our case, we have two important differences. First, the path is given implicitly by the gradient field, and not explicitly. Second, the cart can follow any flow line, and not one determined a priori.

Since the system has two inputs ($\omega$ and $\nu$), we will decompose our analysis in two parts as follows.
Define $\theta_d$ and $\nu_d$, the desired angle and the desired linear velocity, as
\begin{align}
  \label{eq:asymptotic-theta}
  \theta_d&=\arctan_2 (g_2,g_1)\\
  \label{eq:asymptotic-nu}
  \nu_d&=\norm{g},
\end{align}
where $\arctan_2$ is the four-quadrant inverse tangent function and $g=\bmat{g_1\; g_2}\transpose=-\gradx\varphi$.

Then, from \eqref{eq:unicycle-model}, condition \eqref{eq:asymptotic-xdot} is equivalent to
\begin{align}
  \theta=\theta_d, &&  \nu=\nu_d && \forall t>T.
\end{align}
\subsection{Tracking $\theta_d$}
In this section, we assume that $\nu$ is given, and focus on specifying $\omega$ to achieve~\eqref{eq:asymptotic-theta}. Define the signed angle error
\begin{equation}
\theta_e=\mod(\theta-\theta_d+\pi,2\pi)-\pi
\end{equation}
and the Lyapunov function $e=\frac{1}{2}\theta_e^2$. Note that $\theta_e \in [-\pi,\pi)$.

The derivative of $e$ is given by
\begin{equation}
  \label{eq:der-err}
  \dot{e}=\theta_e (\dot{\theta}-\dot{\theta}_d)=\theta_e (\omega-\dot{\theta}_d)
\end{equation}
where $\dtheta_d$ is given by\footnote{Recall that
  $\frac{\partial \arctan_2(x,y)}{\partial x}= -\frac{y}{x^2+y^2}$ and $\frac{\partial \arctan_2(x,y)}{\partial y}= \frac{x}{x^2+y^2}$}
\begin{multline}
  \label{eq:dthetad}
  \dot{\theta}_d%=\frac{1}{\norm{g}^2}(- g_1 \dot{g}_2+ g_2\dot{g}_1)
= \frac{1}{\norm{g}^2}\bmat{g_2\\-g_1}\transpose \dot{g}=\frac{1}{\norm{g}^2}g\transpose\bmat{0 & 1\\ -1 &0} \dot{g}\\
= \sum_{i=1}^N \ri\inverse \frac{g\transpose S H_i\dot{x}}{\norm{g}^2} %=\sum_{i=1}^N \ri\inverse a_i 
=\sum_{i:a_i>0} \ri\inverse \abs{a_i} -\sum_{i:a_i<0} \ri\inverse \abs{a_i} 
\end{multline}
and where 
\begin{align}
  S=\bmat{0 & 1\\ -1 & 0},&&
  a_i= \frac{g\transpose S H_i\dot{x}}{\norm{g}^2},&&
  \dot{x}=\bmat{\cos(\theta)\\\sin(\theta)}\nu.
\end{align}
Note that the quantities $\{a_i\}$ are directly proportional to the specified linear velocity $\nu$.
If we could compute $\dtheta_d$ exactly, we could simply choose a control $\omega$ in \eqref{eq:der-err} yielding $\dot{e}<0$ and obtain convergence. However, $\dtheta_d$ requires the unknown range information $\{\ri\}$. As an alternative, we will show convergence only for the trajectories $x(t)$ of our closed-loop system that satisfy the condition 
\begin{equation}
  \label{eq:condition-rho}
  \ri\inverse\bigl(x(t)\bigr)<\rho
\end{equation}
for all $t$ greater than some $T_{x}$. As it will be clear later, the bound $\rho$ controls a trade-off between the region of convergence and the maximum control effort (a similar condition is also used in \cite{Aranda:AR13}). 

Sliding mode control design follows two phases. In the first phase, we choose a so-called \emph{switching manifold}, which is a sub-manifold of the state space containing the desired trajectories of the system. In our case, this is given by the condition $\theta_e(z)=0$, which specifies the gradient flow trajectories. The second phase is to specify control laws for the two cases $\theta_e(z)<0$ and $\theta_e(z)>0$ which bring the state of the system to the switching manifold in finite time. By switching between these two control laws, the state of the system will be maintained inside the switching manifold.

To construct the two control laws, we will exploit the particular structure of \eqref{eq:dthetad} and the bound \eqref{eq:condition-rho}. In practice, we define the angular velocity control input as
\begin{align}
\label{eq:control-omega}
  \omega&=
  \begin{cases}
    \rho\sum_{i:a_i>0} a_i+k_\theta=\rho\sum_{i:a_i>0} \abs{a_i}+k_\theta &\textrm{if } \theta_e\leq 0,\\
    \rho\sum_{i:a_i<0} a_i-k_\theta=-\rho\sum_{i:a_i<0} \abs{a_i}-k_\theta &\textrm{if } \theta_e>0,\\
  \end{cases}
\end{align}
where $k_\theta>0$ is arbitrary.
% Notice that $\sign(\omega)=-\sign(\theta_e)$, and that (from~\eqref{eq:condition-rho})
% \begin{align}
%   0\leq &d_i \leq \rho\\
%   \implies&\begin{cases} 0\leq a_i d_i \leq a_i \rho & \textrm{if } a_i\geq 0 \\a_i\rho\leq a_i d_i \leq 0 & \textrm{if } a_i\leq 0
%   \end{cases}\\
%   \implies&\begin{cases} 0\leq a_i d_i \leq a_i \rho & \textrm{if } a_i\geq 0 \\a_i\rho\leq a_i d_i \leq 0 & \textrm{if } a_i\leq 0
%   \end{cases},
% \end{align}
% which implies
% \begin{equation}
%   \sum_{i:a_i<0}a_i \rho \leq \sum_i a_i d_i \leq \sum_{i:a_i>0}a_i \rho.
% \end{equation}

Then, by substituting \eqref{eq:control-omega} into \eqref{eq:der-err}, we obtain 
\begin{equation}
  \dot{e}=-\theta_e\bigl(\sum_{i:a_i>0} (\ri\inverse-\rho)\abs{a_i}-k_\theta - \sum_{i:a_i<0} \ri\inverse \abs{a_i}\bigr)
\end{equation}
if $\theta_e\leq 0$ and
\begin{equation}
    \dot{e}=\theta_e\bigl(\sum_{i:a_i<0} (\ri\inverse-\rho) \abs{a_i}-k_\theta - \sum_{i:a_i<0} \ri\inverse \abs{a_i}\bigr)
\end{equation}
if $\theta_e\geq 0$. In both cases we have
\begin{equation}
  \dot{e}=\abs{\theta_e}\dert\abs{\theta_e}\leq -\abs{\theta_e}k_\theta,
  \label{eq:der-err-negative}
\end{equation}
which implies that $\dert\abs{\theta_e}\leq -k_\theta$, and hence $\theta_e=0$ for all $t>T$, where $T=\frac{\abs{\theta_e}}{k_\theta}$  (see \cite{khalil:book02} for details). This will hold for all the trajectories of the closed loop system satisfying \eqref{eq:condition-rho}, and independently from the specific value of $\nu$ (which, however, needs to be known).
% \begin{multline}
%   \label{eq:der-err-negative}
%   \dot{e}=
%   \begin{cases}
%     \theta_e\bigl(\sum_{i:a_i>0} (\rho-\ri\inverse)\abs{a_i}+k_\theta + \sum_{i:a_i<0} \ri\inverse \abs{a_i}\bigr) &\textrm{if } \theta_e\leq 0,\\
%     \theta_e\bigl(\sum_{i:a_i<0} (\ri\inverse-\rho) \abs{a_i}-k_\theta - \sum_{i:a_i<0} \ri\inverse \abs{a_i}\bigr) &\textrm{if } \theta_e>0,\\
%   \end{cases}\\
% =
%   \begin{cases}
%     -\theta_e\bigl(\sum_{i:a_i>0} (\ri\inverse-\rho)\abs{a_i}-k_\theta - \sum_{i:a_i<0} \ri\inverse \abs{a_i}\bigr) &\textrm{if } \theta_e\leq 0,\\
%     \theta_e\bigl(\sum_{i:a_i<0} (\ri\inverse-\rho) \abs{a_i}-k_\theta - \sum_{i:a_i<0} \ri\inverse \abs{a_i}\bigr) &\textrm{if } \theta_e>0,\\
%   \end{cases}\\
%   \leq -\abs{\theta_e}k_\theta.
% \end{multline}
Note that, by increasing $\rho$, we increase the set of trajectories that satisfy condition \eqref{eq:condition-rho}, but we also increase the magnitude of the control in~\eqref{eq:control-omega}.

\subsection{Tracking $\nu_d$}
\label{sec:tracking-nu_d}
In the previous section, we have given a sliding mode controller which ensures convergence of $\theta$ to $\theta_d$ in finite time. %In order to compute the control law, the value of the linear velocity input $\nu$ needs to be known, but the convergence does not depend on its specific value.
Therefore, we can give a simple law such that $\nu=\nu_d$ when $\theta=\theta_d$. We propose the following inner product:
\begin{equation}
  \label{eq:control-nu}
  \nu=-\bmat{\cos(\theta)\;\;\sin(\theta)} g.  
\end{equation}
Note that, with this law, the value of $\nu$ becomes negative when the direction of the unicycle is opposite to the one indicated by the gradient field (\ie, the unicycle can be pushed backward).

\subsection{The closed-loop system}
We can summarize the results above in the following theorem, which represents our second major contribution.
\begin{theorem}\label{thm:convergence-unicycle}
  Given a cost $\varphi$ and an arbitrary parameter $\rho>0$, consider the closed loop system given by the unicycle model \eqref{eq:dynamical-model},~\eqref{eq:unicycle-model} with control inputs \eqref{eq:control-omega},~\eqref{eq:control-nu}. Let $z(t)$ be any trajectory of such system where the first two coordinates $x(t)$ satisfy $\ri\inverse(x(t))<\rho$. Then, $x(t)$ converges in finite time to a curve in the gradient flow of $\varphi$.
\end{theorem}
\medskip

If we the cost function $\varphi$ in \S\ref{sec:control-an-holonomic}, Theorem~\ref{thm:convergence-unicycle} together with Theorem~\ref{thm:convergence-integrator} provides a solution to the visual homing problem for the unicycle model with a region of convergence which can be made arbitrarily large by changing $\rho$. In the simulations, however, we will show that we can practically obtain global convergence with a moderate value for $\rho$.

\subsection{Chattering and input saturation}
\label{sec:chatt-input-satur}
A well-known drawback of sliding mode control is the presence of chattering, where, due to actuation delays (\eg, from a discrete-time implementation of the controller) the state of the system will not exactly fall into the sliding manifold, but will overshoot it, and the closed loop system ends up switching between the two sides at high frequency. Standard ways to reduce the effects of chattering include augmenting the sliding mode controller with integrators or approximating the switching law with a ``continuous'', high-gain version \cite{khalil:book02}. For our particular application, however, we need also to be able to track the landmarks in the images from the camera as the robot moves. If the field of view of the camera is restricted, this might become impossible if the unicycle turns too fast. In addition, there might be physical constraints on the maximum speed at which the unicycle can move. Therefore, our control law should take into account constraints of the kind $\abs{\nu}\leq \numax$ and $\abs{\omega}\leq\omegamax$. Fortunately, our control laws are directly proportional to the magnitude of the gradient, which tends asymptotically to zero. Therefore, at least in the limit, our control naturally satisfies these constraints. Potential problems can only arise in the transient phase. We propose the following substitution:
\begin{equation}
  \bmat{\nu\\\omega}\gets  \bmat{\nu\\\omega}\min\{1,\abs{\frac{\numax}{\nu}},\abs{\frac{\omegamax}{\omega}}\}.
\end{equation}
Intuitively, this amounts to, when necessary, reduce the speed at which the unicycle tracks the gradient flow lines. Notice that, since the constraints tends to be satisfied naturally as the robot approaches the destination, this substitution does not change the asymptotic behavior of the controller.
This substitution has also the benefit of lessening the effects of chattering (the state cannot overshoot the sliding manifold too much).% and allowing to increase $\rho$ in \eqref{eq:condition-rho} (thus arbitrarily enlarging the region of convergence, with the caveat that the unicycle could move at lower speeds than strictly necessary).

\section{SENSITIVITY ANALYSIS}
\label{sec:sensitivity-analysis}
Until now, we have considered the ideal case where bearing directions at the home location, $\{\betagi\}$, are perfectly known. In practice, these bearings are measured, and might contain noise. Following our optimization framework, this results in a deviation of the minimizer of $\varphi$ from the ideal location $x_g$. Our goal is to relate the magnitude of this deviation and with the magnitude of the noise in $\{\betagi\}$. This will be used in \S\ref{sec:simulations} to build heat maps indicating what locations are ``harder'' to home into for a given set of landmarks.

To start our analysis, let $\{\betagi^\ast\}$ be a set of desired bearing vectors and denote as $x^\ast$ a minimizer of $\varphi(x,\betagi^\ast)$ (note that we made explicit the dependency of $\varphi$ on the desired bearings). From Lemma~\ref{lemma:global-minimum}, we know that when $\{\betagi^\ast\}=\{\betagi\}$, then $x^\ast=x_g$. Formally, our goal is to quantify the variance of the perturbation in $x^\ast$ when $\{\betagi^\ast\}$ is perturbed around $\{\betagi\}$.
Since $x^\ast$ is a minimizer, we have
\begin{equation}
  \label{eq:first-order-condition}
  \gradx\varphi(x^\ast,\{\betagi^\ast\})=0.  
\end{equation}
Using a first-order Taylor expansion of \eqref{eq:first-order-condition}, % with respect to all the arguments of $\varphi$,
 we obtain
\begin{equation}
  \label{eq:der-first-order-condition}
  \Dxgradx\varphi(x^\ast,\betagi^\ast) v_x+\sum_{i=1}^N \Dbetagradx\varphi(x^\ast,\betagi^\ast) v_{\betagi}=0,
\end{equation}
where $v_x$ and $\{v_{\betagi}\}$ denote the perturbation in $x^\ast$ and $\{\betagi^\ast\}$, respectively. Note that, since $\betagi^\ast$ is constrained to have norm one, $v_{\betagi}$ will be always orthogonal to $\betagi^\ast$. We model each vector $v_{\betagi}$ as a i.i.d., zero-mean random variable with covariance matrix $E[v_{\betagi}v_{\betagi}\transpose]=\sigma_\beta^2 (I-\betagi^\ast\betagi^\ast{}\transpose)$, where $\sigma_\beta$ represents the noise level in the perturbation (notice that this matrix is singular in the direction $\betagi^\ast$, due to the orthogonality of $v_{\betagi}$ and $\betagi^\ast$). %Here the covariance should be low-rank, because \betagi is on a manifold
 Then, we can evaluate \eqref{eq:der-first-order-condition} at $(x_g,\betagi)$, and obtain an expression for the covariance matrix of $v_x$, that is
\begin{equation}
  E[v_xv_x\transpose]=  \sigma_\beta^2J_x\inverse \bigl(\sum_{i=1}^N J_{\betai}J_{\betai}\transpose \bigr)J_x^{-T},
\end{equation} 
where $J_x=\Dxgradx\varphi$ and $J_{\betai}=\Dbetagradx\varphi$. In practice, an expression for $J_x$ is given by Proposition~\ref{prop:grad-and-Dgrad}, while, for $J_{\betai}$, we have the following (the proof is similar to the one of Proposition~\ref{prop:grad-and-Dgrad}%
\ifthenelse{\boolean{isconference}}{}{%
, see \S\ref{sec:proof-prop-Dbetagrad} of the Appendix}).
\begin{proposition}
  \label{prop:Dbetagrad}
  The differential $\Dbetagradx$ is given by
  \begin{equation}
    \label{eq:Dbetagrad}
    \Dbetagradx=-\df(c_i)I-\ddf(c_i)(I-\betai\betai\transpose)\betagi\betai\transpose.
  \end{equation}
\end{proposition}

In order to obtain a more compact representation of the effect of noise, we will condense the covariance matrix into a single number, the variance, using $\sigma_x^2=\trace\bigl(E[v_xv_x\transpose]\bigr)$.

The ratio $\frac{\sigma_x^2}{\sigma_y^2}$ gives a measure of the robustness of the minimizer of $\varphi$ with respect to the noise in the bearings. % $\{\betagi\}$.% In practice, this depends on the relative location of the goal position and of the landmarks.

\section{SIMULATIONS}
\label{sec:simulations}

In this section we illustrate the behavior of the proposed control laws through simulations. We generate a scenario where we place $10$ landmarks at random in a square area % with side $L=20$
 and a target location in the lower left quadrant. We first test the model considered in \S\ref{sec:control-an-holonomic}. We initialize the system from $9$ different locations uniformly distributed in the square area, and record the trajectory obtained by integrating the closed-loop system. Figure~\ref{fig:trajectoriesIntegrator} show the results obtained using the functions \eqref{eq:f-cosine}--\eqref{eq:f-angleSq}, together with the level curves and gradient field of the cost function $\varphi$. As one can see, the system converges asymptotically to the home location $x_g$ from all the initial conditions. The trajectories obtained using \eqref{eq:f-cosine} appear to be slightly more direct than those obtained with~\eqref{eq:f-angleSq}, however the differences are quite minor.

 In the second set of experiments we use the same conditions (landmarks, starting and home locations; cost functions), but this time we test the model considered in \S\ref{sec:sliding-mode-unicycle}. We test four initial headings for each starting location. The results are shown in Figure~\ref{fig:trajectoriesUnicycle}. As before, the model converges asymptotically to the home location from any initial condition. When the initial direction of the cart is against the negative gradient, the controller automatically performs a maneuver where the unicycle first goes backward while rotating, and then switches direction to move toward the destination. Notice also how each trajectory converges to a gradient flow line (\ie, the tangent of the trajectory approaches the gradient).

\begin{figure}
  \centering
  \subfloat[Cosine, equation~\eqref{eq:f-cosine}]{\includegraphics[clip,trim=0 0 0 1mm]{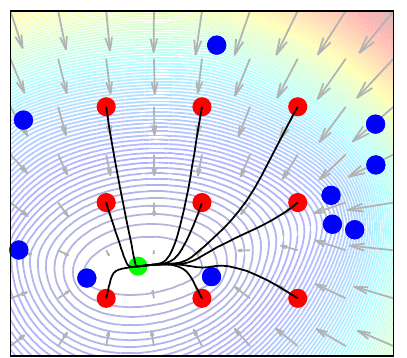}}
  \subfloat[Angle squared, equation~\eqref{eq:f-angleSq}]{\includegraphics[clip,trim=0 0 0 1mm]{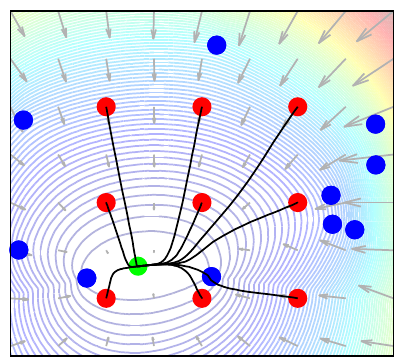}}
  \caption{Trajectories for the integrator model of \S\ref{sec:control-an-holonomic} with our controller and different choices for $f$. We also superimpose the level sets and gradient field of the cost $\varphi$.}
  \label{fig:trajectoriesIntegrator}
\end{figure}
% \begin{figure}
%   \centering
% \includegraphics{figures/trajectoriesUnicycleCosine}
%   \caption{Trajectories for the unicycle model of \S\ref{sec:sliding-mode-unicycle} with our controller with different choices for $f$. The four different colors represent, for each initial point, four initial heading directions: north (blue).}
%   \label{fig:trajectoriesUnicycle}
% \end{figure}
\begin{figure}
  \centering
  \subfloat[Cosine, equation~\eqref{eq:f-cosine}]{\includegraphics[clip,trim=0 0 0 1mm]{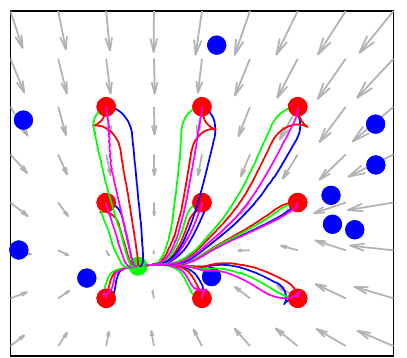}}
  \subfloat[Angle squared, equation~\eqref{eq:f-angleSq}]{\includegraphics[clip,trim=0 0 0 1mm]{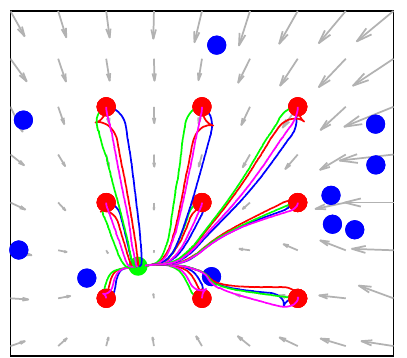}}
  \caption{Trajectories for the unicycle model of \S\ref{sec:sliding-mode-unicycle} with our controller. The four different colors represent the four initial heading directions: north (red), east (blue), south (pink), west (green). The cost used is the same as in Figure~\ref{fig:trajectoriesIntegrator}.}
  \label{fig:trajectoriesUnicycle}
\end{figure}
\begin{figure}
  \centering
\includegraphics{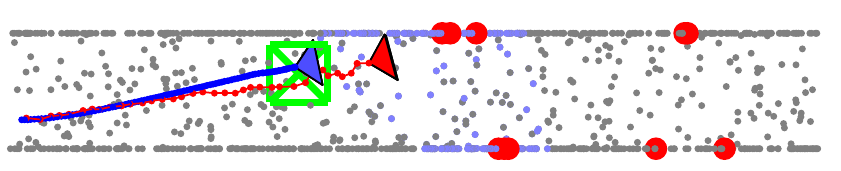}
  \caption{Simulation of a complete system. Grey and blue dots: points used to generate features for the visual odometry. Red dots: landmarks. Green box: home location. Blue triangle and line: actual pose and trajectory of the unicycle. Red triangle and line: pose and trajectory estimated by visual odometry.}
  \label{fig:completeSystem}
\end{figure}

In the third experiment, we simulate a complete system in a corridor-like environment, as shown in Figure~\ref{fig:completeSystem}. We use LIBVISO2 \cite{Geiger:IV11} for visual odometry to compute the global compass direction. The grey and light blue points are used to generate features for visual odometry (the light blue points are those visible by the camera), while the red points represent the landmarks. The home location is given by the green box. The blue triangle and line represent the actual final pose and the actual trajectory of the unicycle. The red triangle and line represent those estimated by LIBVISO2. Notice how, even if the estimation of the translation from visual data fails, our controller can still drive the unicycle to the correct location.

Finally, we demonstrate the use of the sensitivity analysis from \S\ref{sec:simulations}. In Figure~\ref{fig:heatmaps} we show the heat maps of $\log\left(\frac{\sigma_x^2}{\sigma_y^2}\right)$ corresponding to the configuration of landmarks of Figures~\ref{fig:trajectoriesIntegrator}--\ref{fig:trajectoriesUnicycle}, together with two other simple configurations. In addition, we show how the heat maps change when the coordinate of the bearings are scaled along one of the axes. The reshaping function used is \eqref{eq:f-cosine}.
Generally speaking, it seems to be hard to qualitatively predict the aspect of the heat map by simply looking at the bearing configuration. However, we can notice some trends. First and foremost, as one would expect, the stability of the homing position quickly decreases outside a neighborhood of the landmarks. We can notice that the homing positions evenly surrounded by the landmarks are among the most stable, as one would expect. However, a little bit surprisingly, close to each landmark there are areas where the stability is significantly increased or decreased. Finally, we can notice that, as the points tend to be collinear, the stability of home locations along the same line decreases.

\begin{figure}
  \newcommand{\figspace}{1cm}
  \centering
  \subfloat[Configuration from Figures~\ref{fig:trajectoriesIntegrator}--\ref{fig:trajectoriesUnicycle}]{\hspace{\figspace}\includegraphics[clip,trim=0 0 0 1mm]{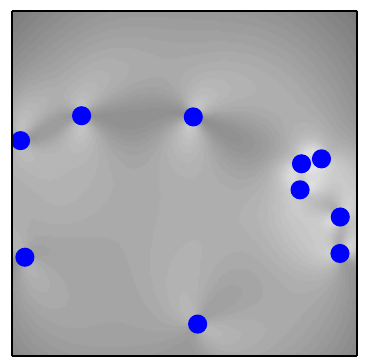}\includegraphics[clip,trim=0 0 0 1mm]{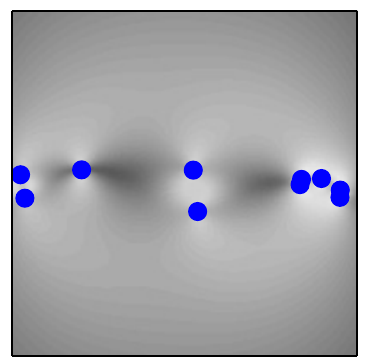}\hspace{\figspace}}\\
  \subfloat[Square configuration]{\hspace{\figspace}\includegraphics[clip,trim=0 0 0 1mm]{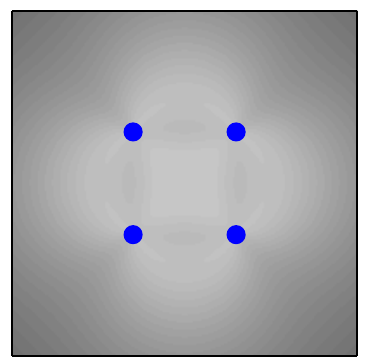}\includegraphics[clip,trim=0 0 0 1mm]{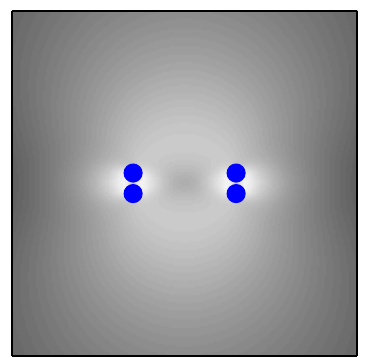}\hspace{\figspace}}\\
  \subfloat[Triangle configuration]{\hspace{\figspace}\includegraphics[clip,trim=0 0 0 1mm]{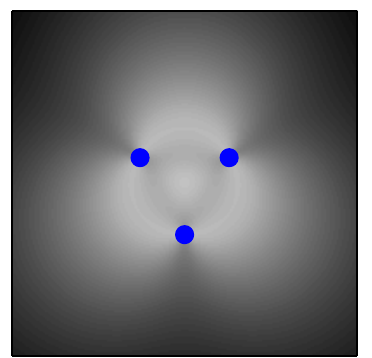}\includegraphics[clip,trim=0 0 0 1mm]{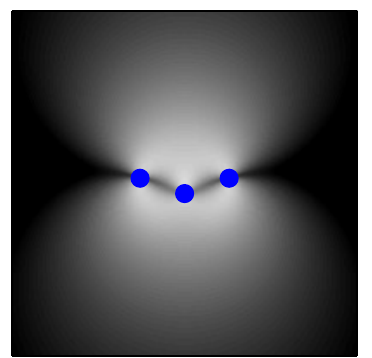}\hspace{\figspace}}
  \caption{Examples of heat maps of $\log\left(\frac{\sigma_x^2}{\sigma_y^2}\right)$ for various bearing configurations. White: low sensitivity. Black: high sensitivity.}
  \label{fig:heatmaps}
\end{figure}

\section{CONCLUSIONS}
We have presented a new approach to the visual homing problem using bearing-only measurements, where the control law is given by the gradient of a cost function and global convergence is guaranteed. We also developed a sliding mode controller for a unicycle model which follows the gradient field line with limited information (lower and upper bounds on the differential of the gradient). When coupled with the previous control law, this provides a natural solution to the bearing-only visual homing problem for the unicycle model. Our framework allows us also to compute heat maps indicating the relative sensitivity of different homing location with respect to noise in the home bearing directions.

 In future work, we plan to study more throughly the effect of the reshaping function $f$ in $\varphi$, especially for question regarding robustness to outliers. We will also investigate a compass-free extension, where the relative rotation between the current and goal reference frames is estimated through epipolar geometry and it is used in the cost $\varphi$. Finally, it would be interesting to see whether the same function $\varphi$ could be used in the framework of \cite{Cowan:TRA02,Koditschek:book89} allowing the control of a unicycle with a Lagrangian model.

\bibliographystyle{IEEEtran}
\newcommand{\bibpath}{.}
\bibliography{%
\bibpath/control,%
\bibpath/IEEEfull,%
\bibpath/IEEEConfFull,%
\bibpath/OtherFull,%
\bibpath/visualHoming,%
\bibpath/visualOdometry,%
\bibpath/tron}

\begin{thebibliography}{10}
\providecommand{\url}[1]{#1}
\csname url@rmstyle\endcsname
\providecommand{\newblock}{\relax}
\providecommand{\bibinfo}[2]{#2}
\providecommand\BIBentrySTDinterwordspacing{\spaceskip=0pt\relax}
\providecommand\BIBentryALTinterwordstretchfactor{4}
\providecommand\BIBentryALTinterwordspacing{\spaceskip=\fontdimen2\font plus
\BIBentryALTinterwordstretchfactor\fontdimen3\font minus
  \fontdimen4\font\relax}
\providecommand\BIBforeignlanguage[2]{{%
\expandafter\ifx\csname l@#1\endcsname\relax
\typeout{** WARNING: IEEEtran.bst: No hyphenation pattern has been}%
\typeout{** loaded for the language `#1'. Using the pattern for}%
\typeout{** the default language instead.}%
\else
\language=\csname l@#1\endcsname
\fi
#2}}

\bibitem{Tron:ICRA14}
R.~Tron and K.~Daniilidis, ``An optimization approach to bearing-only visual
  homing with applications to a 2-{D} unicycle model,'' in \emph{{IEEE}
  International Conference on Robotics and Automation}, 2014.

\bibitem{Cartwright:JCP83}
B.~A. Cartwright and T.~S. Collett, ``Landmark learning in bees,''
  \emph{Journal of Computational Physiology A}, vol. 151, no.~4, pp. 521--543,
  1983.

\bibitem{Wehner:JCPA03}
R.~Wehner, ``Desert ant navigation: how miniature brains solve complex tasks,''
  \emph{Journal of Computational Physiology A}, vol. 189, no.~8, pp. 579--588,
  2003.

\bibitem{McLeman:IS02}
M.~A. McLeman, S.~C. Pratt, and N.~R. Franks, ``Navigation using visual
  landmarks by the ant {L}eptothorax albipennis,'' \emph{Insectes Sociaux},
  vol.~49, pp. 203 -- 208, 2002.

\bibitem{Hong:CSM92}
J.~Hong, X.~Tan, B.~Pinette, R.~Weiss, and E.~M. Riseman, ``Image-based
  homing,'' \emph{{IEEE} Control Systems Magazine}, vol.~12, no.~1, pp. 38--45,
  1992.

\bibitem{Lambrinos:GNC98}
D.~Lambrinos, R.~Moller, R.~Pfeifer, and R.~Wehner, ``Landmark navigation
  without snapshots: the average landmark vector model,'' in \emph{Goettingen
  Neurobiology Conference}, vol.~1, 1998, p. 221.

\bibitem{Moller:CVMRW98}
R.~M\:oller, D.~Lambrinos, R.~Pfeifer, and R.~Wehner, ``Insect strategies of
  visual homing in mobile robots,'' 1998.

\bibitem{Schroeter:IAS08}
D.~Schroeter and P.~Newman, ``On the robustness of visual homing under landmark
  uncertainty,'' in \emph{International Conference on Autonomous Systems},
  2008.

\bibitem{Goedeme:CIRA05}
T.~Goedeme, T.~Tuytelaars, L.~V. Gool, D.~Vanhooydonck, E.~Demeester, and
  M.~Nuttin, ``Is structure needed for omnidirectional visual homing?'' in
  \emph{{IEEE} International Symposium on Computational Intelligence and
  Robotic Automation}, 2005, pp. 303--308.

\bibitem{Argyros:CVPR01}
A.~A. Argyros, K.~E. Bekris, and S.~E. Orphanoudakis, ``Robot homing based on
  corner tracking in a sequence of panoramic images,'' in \emph{{IEEE}
  Conference on Computer Vision and Pattern Recognition}, vol.~2, 2001, pp.
  3--11.

\bibitem{Lim:RSS09}
J.~Lim and N.~Barnes, ``Robust visual homing with landmark angles,'' in
  \emph{Robotics: Science and Systems}, 2009.

\bibitem{Liu:ICRA13}
M.~Liu, C.~Pradalier, Q.~Chen, and R.~Siegwart, ``A bearing-only 2d/3d-homing
  method under a visual servoing framework,'' in \emph{{IEEE} International
  Conference on Robotics and Automation}, 2010, pp. 4062--4067.

\bibitem{Loizou:CDC07}
S.~G. Loizou and V.~Kumar, ``Biologically inspired bearing-only navigation and
  tracking,'' in \emph{{IEEE} International Conference on Decision and
  Control}, 2007, pp. 1386--1391.

\bibitem{Corke:RR03}
P.~Corke, ``Mobile robot navigation as a planar visual servoing problem,'' in
  \emph{Robotic Research}, ser. Springer Tracts in Advanced Robotics.\hskip 1em
  plus 0.5em minus 0.4em\relax Springer Berlin Heidelberg, 2003, vol.~6, pp.
  361--372.

\bibitem{Papanikolopoulos:TAC93}
N.~P. Papanikolopoulos and P.~K. Khosla, ``Adaptive robotic visual tracking:
  theory and experiments,'' \emph{{IEEE} Transactions on Automatic Control},
  vol.~38, no.~3, pp. 429--445, 1993.

\bibitem{Liu:ICRA12}
M.~Liu, C.~Pradalier, F.~Pomerleau, and R.~Siegwart, ``Scale-only visual homing
  from an omnidirectional camera,'' in \emph{{IEEE} International Conference on
  Robotics and Automation}, 2012, pp. 3944--3949.

\bibitem{Cowan:TRA02}
N.~Cowan, J.~D. Weingarten, and D.~E. Koditschek, ``Visual servoing via
  navigation functions,'' \emph{{IEEE} Journal of Robotics and Automation},
  vol.~18, no.~4, pp. 521--533, 2002.

\bibitem{Koditschek:book89}
D.~E. Koditschek \emph{et~al.}, ``The application of total energy as a lyapunov
  function for mechanical control systems,'' \emph{Dynamics and control of
  multibody systems}, vol.~97, pp. 131--157, 1989.

\bibitem{Basri:ICCV98}
R.~Basri, E.~Rivlin, and I.~Shimshoni, ``Visual homing: surfing on the
  epipoles,'' in \emph{{IEEE} International Conference on Computer Vision},
  1998, pp. 863--869.

\bibitem{Piazzi:CPR03}
J.~Piazzi, D.~Prattichizzo, and A.~Vicino, ``Visual servoing along epipoles,''
  in \emph{Control Problems in Robotics}, ser. Springer Tracts in Advanced
  Robotics.\hskip 1em plus 0.5em minus 0.4em\relax Springer Berlin Heidelberg,
  2003, vol.~4, pp. 215--231.

\bibitem{Rives:IROS00}
P.~Rives, ``Visual servoing based on epipolar geometry,'' in \emph{{IEEE}
  International Conference on Intelligent Robots and Systems}, vol.~1, 2000,
  pp. 602--607.

\bibitem{Usher:ICRA03}
K.~Usher, P.~Ridley, and P.~Corke, ``Visual servoing of a car-like vehicle - an
  application of omnidirectional vision,'' in \emph{{IEEE} International
  Conference on Robotics and Automation}, 2003.

\bibitem{Wei:IJISTA05}
R.~Wei, D.~Austin, and R.~Mahony, ``Biomimetic application of desert ant visual
  navigation for mobile robot docking with weighted landmarks,''
  \emph{International Journal of Intelligent Systems Technologies and
  Applications}, vol.~1, no.~1, pp. 174--190, 2005.

\bibitem{Mariottini:ICRA04}
G.~Mariottini, D.~Prattichizzo, and G.~Oriolo, ``Epipole-based visual servoing
  for nonholonomic mobile robots,'' in \emph{{IEEE} International Conference on
  Robotics and Automation}, vol.~1, 2004, pp. 497--503.

\bibitem{Mariottini:TRA07}
G.~L. Mariottini, G.~Oriolo, and D.~Prattichizzo, ``Image-based visual servoing
  for nonholonomic mobile robots using epipolar geometry,'' \emph{{IEEE}
  Transactions on Robotics and Automation}, vol.~23, no.~1, pp. 87--100, 2007.

\bibitem{LopezNicolas:ICRA06}
G.~L\'{o}pez-Nicol\'{a}s, C.~Sag\"{u}\'{e}s, J.~J. Guerrero, D.~Kragic, and
  P.~Jensfelt, ``Nonholonomic epipolar visual servoing,'' in \emph{{IEEE}
  International Conference on Robotics and Automation}, 2006.

\bibitem{Becerra:IROS08}
H.~M. Becerra and C.~Sagues, ``A sliding mode control law for epipolar visual
  servoing of differential-drive robots,'' in \emph{{IEEE} International
  Conference on Intelligent Robots and Systems}, 2008, pp. 3058--3063.

\bibitem{Aranda:AR13}
M.~Aranda, G.~L{\'o}pez-Nicol{\'a}s, and C.~Sag{\"u}{\'e}s, ``Angle-based
  homing from a reference image set using the 1d trifocal tensor,''
  \emph{Autonomous Robots}, vol.~34, no. 1-2, pp. 73--91, 2013.

\bibitem{khalil:book02}
H.~K. Khalil, \emph{Nonlinear systems}.\hskip 1em plus 0.5em minus 0.4em\relax
  Prentice hall {U}pper {S}addle {R}iver, 2002, vol.~3.

\bibitem{Geiger:IV11}
A.~Geiger, J.~Ziegler, and C.~Stiller, ``Stereo{S}can: Dense 3d reconstruction
  in real-time,'' in \emph{Intelligent Vehicles Symposium}, 2011.

\end{thebibliography}

\appendix
\subsection{Proof of Proposition~\ref{prop:grad-and-Dgrad}}
\label{sec:proof-gradient}
%In this section we will give a proof of Proposition~\ref{prop:grad-and-Dgrad}.
We first compute the gradient by evaluating the function along a line $\tx$ and using the definition~\eqref{eq:definition-differential}. We will need the following derivatives (recall that $\dtx=v$).
\begin{align}
\dtri&=\dert\norm{x_i-\tx}=\frac{1}{2\norm{x_i-\tx}}\dert\norm{x_i-\tx}^2\nonumber\\
%       =\frac{\tri\inverse}{2}\dert\norm{x_i-\tx}^2\\ =\frac{\tri\inverse}{2}\dert\bigl((x_i-\tx)\transpose(x_i-\tx)\bigr)\\
&=-\tri\inverse(x_i-\tx) \transpose v= -v\transpose \tbetai,\\
\dert\tri\inverse&=-\tri^{-2}\dtri=\tri^{-2}v\transpose \tbetai,\\
\dtbetai&=\dert \tri\inverse (x_i-\tx)%=\tri^{-2}(x_i-\tx)\tbetai\transpose v-\tri\inverse v\\
=-\tri\inverse P_{\betai}(I-\tbetai\tbetai\transpose) v,\\
\dtci&=\betagi\transpose \dtbetai%=-\tri\inverse v\transpose(I-\tbetai\tbetai\transpose)\betagi
=-\tri\inverse v\transpose P_{\betai}\betagi. \label{eq:directional-derivative-c}
\end{align}
Hence
\begin{equation}
  \dtvarphii\!=\!f(\tci)\dtri+\df(\tci)\ri\dtci%\\=-f(\tci) v\transpose\tbetai -\df(\tci)v\transpose(I-\betai\betai\transpose)\betagi\\
  \!=\!v\transpose\!\bigl(-f(\tci) \tbetai - \df(\tci)P_{\betai}\betagi \bigr),
\end{equation}
from which \eqref{eq:gradient-cost} follows.

By reorganizing terms, another expression for $\gradx\varphii$ is
\begin{equation}
  \gradx\varphii%= -f(c_i) \betai - \df(c_i) (I-\tbetai\tbetai\transpose) \betagi\\
= (\df(\ci)\ci-f(c_i)) \betai - \df(c_i) \betagi.\label{eq:cost-grad-alternate}
\end{equation}

For $\Dxgradx\varphi$, we will again use definition \eqref{eq:definition-differential}.
%See POCDgradBearingCost.m for a numerical verification of the derivatives
\begin{multline}
\dert \gradx \varphi_i = (\ddf(\tci)\tci+\df(\tci)-\df(\tci))\tbetai\dtci\\+(\df(\tci)\tci-f(\tci))\dtbetai-\ddf(\tci)\betagi\dtci\\
%=\ddf(\tci) (\tci\tbetai-\betagi)\dtci+(\df(\tci)\tci-f(\tci))\dtbetai\\
%=\bigl(\ddf(\tci) (\tci\tbetai-\betagi)\betagi\transpose+(\df(\tci)\tci-f(\tci))I\bigr)\dtbetai\\
=-\tri\inverse\bigl(\ddf(\tci) (\tci\tbetai-\betagi)\betagi\transpose+(\df(\tci)\tci-f(\tci))I\bigr) P_{\betai} v
\end{multline}

\subsection{Proof of Lemma~\ref{lemma:derivative-lines}}
\label{sec:proof-derivative-lines}
%The proof of the lemma is articulated in several steps. First, we provide a change of coordinates for the domain of $\varphii$ which allows us to make some simplifbetaing assumptions without any loss of generality. Then, we use these assumptions to explicitly compute expressions for quantities appearing in $\dtvarphii$. Finally, we use these expressions to show each part of the lemma.

\myparagraph{Step 1) Change of coordinates}
 To make the derivations simpler, we make a few coordinate transformations which do not change the value of $\tvarphii$. First, the values of $\tri$, $\tbetai$ and $\tci$ are invariant to a global translation. %, \ie, they do not change if we operate the transformation $x_i\gets x_i-o$ and $\tx\gets \tx-o$, where $o\in \real{d}$ is arbitrary.
 Hence, w.l.o.g, we assume $x_i=0$, \ie, we center our coordinates on the $i$-th landmark. Then, we make a global change of scale such that $\norm{x_g}=1$ (this is equivalent to multiplying $\varphi$ by a positive constant, which can be simply undone at the end). Finally, $\ri$ and $\ci$ are invariant to global rotations. 
% With this assumption we have
%   \begin{align}
%   \tri&=\norm{\tx}\\
%   \tbetai&=-\frac{\tx}{\norm{\tx}} \\
%   \betagi&=-\frac{x_g}{\norm{x_g}}\\
%   \tci&=\frac{x_g\transpose \tx}{\norm{x_g}\norm{\tx}}
% \end{align}
% From these expressions, notice that the values of $\tri$ and $\tci$ are invariant to a common rotation and scaling of $\tx$ and $x_g$. In particular, we can choose a scaling such that $\norm{x_g}=1$, and
Hence, we choose a rotation which aligns $x_g$ with the first axis and such that $x_g$ and $d$ span the subspace given by the first and second coordinates. We can then restrict our attention to only the first two coordinates (2-D case) and we assume $x_g=\bmat{1\;\;0}\transpose$.

\myparagraph{Step 2) Explicit expressions}
Let $v=\bmat{v_1\;v_2}\transpose$. Under the previous change of coordinates, we have the following:% explicit formulae:
\begin{align}
  \tx&=\bmat{1+tv_1\;tv_2}\transpose,\\
  \tri&=\sqrt{(1+tv_1)^2+t^2v_2^2},\\
  \tbetai&=-\frac{1}{\sqrt{(1+tv_1)^2+t^2v_2^2}}\bmat{1+tv_1\\tv_2},\\
  \betagi&=-\bmat{1\;0}\transpose,\\
  \tci&=\frac{1+tv_1}{\sqrt{(1+tv_1)^2+t^2v_2^2}},\\
  \dtci&=-\frac{tv_2^2}{\sqrt{(1+tv_1)^2+t^2v_2^2}^3}  \begin{cases}
    =0 & \textrm{for } t=0,\\
    <0 & \textrm{for } t>0,\\
  \end{cases}
\end{align}
\begin{align}
  v\transpose\betagi&=-v_1 \label{eq:vbetagi},\\
  v\transpose\tbetai&=-\frac{v_1(1+tv_1)+tv^2_2}{\sqrt{(1+tv_1)^2+t^2v_2^2}}, \label{eq:vbetai}\\
  v\transpose &P_{\betai}\betagi=-\tri\dtci\nonumber\\&=\frac{t v_2^2}{(1+tv_1)^2+t^2v_2^2}  \begin{cases}
    =0 & \textrm{for } t=0,\\
    >0 & \textrm{for } t>0.\\
  \end{cases} \label{eq:vPbetaibetagi}
\end{align}
Note that \eqref{eq:vPbetaibetagi} is follows from \eqref{eq:directional-derivative-c}.
Also, from \eqref{eq:vbetai} we have
\begin{align}
    v\transpose\tbetai(0)&=-v_1,\\
    \dert v\transpose\tbetai &= -\frac{v_2^2}{\sqrt{(1+tv_1)^2+t^2v_2^2}^3}<0,
\end{align}
which implies that $v\transpose\tbetai$ is monotonically decreasing for $t\geq 0$, with maximum $-v_1$ at $t=0$. Hence, from \eqref{eq:vbetagi},
\begin{equation}
    v\transpose\tbetai\leq v\transpose\betagi. \label{eq:vtbetai-vbetagi}
\end{equation}

\myparagraph{Step 3) Proof of the bounds}
 In order to avoid some negative signs, we will work with $-\dtvarphi$.
From~\eqref{eq:gradient-cost} and \eqref{eq:cost-grad-alternate}:
  \begin{align}
  -\dtvarphii&%=-v\transpose\gradx\varphii(\tx) \\
  =f(\tci) v\transpose\tbetai + \df(\tci) v\transpose(I-\tbetai\tbetai\transpose) \betagi \label{eq:negative-directional-derivative}\\
  %   &=f(\tci) v\transpose\tbetai + \df(\tci) v\transpose\betagi -\df(\tci)v\transpose\tbetai\tbetai\transpose \betagi\\
  %   &
&=\bigl(f(\tci) -\df(\tci)\tci\bigr) v\transpose\tbetai + \df(\tci) v\transpose\betagi. \label{eq:negative-directional-derivative-2}
  \end{align}

First, consider the case where $t=0$. Since $v\transpose\tbetai(0)=v\transpose\betagi$, $\tci(0)=1$, $f\bigl(\tci(0)\bigr)=0$, and $-\dtvarphii$ reduces to
\begin{equation}
  -\dtvarphii(0)=-\df\bigl(\tci(0)\bigr) v\transpose\betagi +\df\bigl(\tci(0)\bigr) v\transpose\betagi=0.
\end{equation}
This difference is well defined, since $\df(1)$ is finite (see \eqref{eq:property-f-firstderivative}).

Now, consider the case $t>0$, and assume that $v$ is not parallel to $\betagi$. We have two cases depending on $\sign(v\transpose\tbetai)$.

1) Assume $v\transpose\tbetai<0$. Intuitively, this condition indicates that $\tx$ moves away from $x_i$ as $t$ increases. Combining the last assumption with the non-negativity of $f$, and property \eqref{eq:property-f-firstderivative} with \eqref{eq:vPbetaibetagi}, we see that \eqref{eq:negative-directional-derivative} is negative.

2) Assume $v\transpose\tbetai>0$. Intuitively, this condition indicates that $\tx$ gets closer to $x_i$ as $t$ increases. 
Note that $v\transpose\tbetai>0$ implies, from \eqref{eq:vtbetai-vbetagi}, that $v\transpose\betagi>0$. We will consider two further subcases depending on $\sign\bigl(f(\tci) -\df(\tci)\tci\bigr)$. 

2a) Assume $\bigl(f(\tci) -\df(\tci)\tci\bigr)<0$. Combining this with the positivity of $v\transpose\tbetai$, $v\transpose\betagi$ and with property \eqref{eq:property-f-firstderivative}, we have that \eqref{eq:negative-directional-derivative-2} is negative.

2b) Assume $\bigl(f(\tci) -\df(\tci)\tci\bigr)\geq0$. Then, using \eqref{eq:vtbetai-vbetagi} in \eqref{eq:negative-directional-derivative-2}, we have
  \begin{equation}
    -\dtvarphii\leq \bigl(f(\tci) +(1-\tci)\df(\tci)\bigr) v\transpose\betagi.
  \end{equation}
From property \eqref{eq:property-f-derivativemix}, this implies $-\dtvarphii<0$.

Now we pass to the case $v=a\betagi$, $a>0$. Let $t_0=\frac{\norm{x_i-x_g}}{\norm{v}}$. We have that $\tx(t_0)=0=x_i$, \ie, the line passes through the landmark. For $t\in[0,t_0)$, we have $\tbetai\equiv \betagi$, $\tci\equiv 1$ and $f(\tci)\equiv0$, which implies $-\dtvarphii\equiv0$ (similarly to the previous case for $t=0$). For $t>t_0$ we have $\tbetai\equiv -\betagi$, and $\tci\equiv -1$, which implies 
\begin{equation}
  -\dtvarphii\equiv -(f(-1)+\df(-1))+\df(-1)\equiv -f(-1)<0.
\end{equation}
Finally, we consider the case $v=a\betagi$, $a<0$. Intuitively, in this case $\tx$ moves exactly opposite $x_i$ as $t$ increases. Similarly to the above, we have $\tbetai\equiv \betagi$, $\tci\equiv 1$ and $f(\tci)\equiv0$, which imply $\dtvarphii\equiv0$.

\ifthenelse{\boolean{isconference}}{\pagebreak}{

\subsection{Proof of Proposition~\ref{prop:radially-unbounded}}
\label{sec:proof-radially-unbounded}
In order to show the claim, we will use one of the definitions of radially unbounded function \cite{khalil:book02}, which requires any level set of $\varphi$, defined as
\begin{equation}
  \Sset(\bar{\varphi})=\{x\in\real{d}:\varphi(x)\leq\bar{\varphi}\},  
\end{equation}
 to be compact. Our strategy is to show that any of these sets is contained in a ball centered at $x_g$ with finite radius $\bar{t}<\infty$, which we denote as $B(x_g,\bar{t})$.
We describe this ball in polar coordinates and use a lower bound of the function $\varphi$ evaluated along lines of the form $\tx(t)=x_g+tv$. For this section, we assume $\norm{v}=1$, so that we have the equivalence $t=\norm{x_g-\tx(t)}$. We need the following preliminary definitions
\begin{align}
  \rgi&=r_i(x_g),\\
  \rgmax&=\max_i\{r_{gi}\},\\
  \ciinf&=\lim_{t\to\infty} y_i\transpose y_g.
\end{align}
Notice that
\begin{multline}
  \ciinf=\lim_{t\to\infty}\betagi\transpose\frac{x_i-x_g-tv}{\norm{x_i-x_g-tv}}=\betagi\transpose v,
\end{multline}
which, together with \eqref{eq:vtbetai-vbetagi} and \eqref{eq:ordering}, implies
\begin{equation}
  \label{eq:ineq-fciinf}
  f\bigl(\tci(t)\bigr) \geq f(\ciinf).
\end{equation}
Also, from the triangular inequality applied to the triangle $\bigl(x_i,x_g,\tx(t)\bigr)$, we have
\begin{equation}
  \label{eq:ineq-tri}
  \tri(t)\geq t-\rgi.  
\end{equation}
Without loss of generality, in this proof we will only consider values of $t$ such that $t>\rgmax$. As it will be clear by the end of the proof, this restriction excludes level sets $\Sset(\bar{\varphi})$ where $\bar{\varphi}$ is too small and that do not contain the ball centered at $x_g$ with radius $\rgmax$, \ie, $B(x_g,\rgmax)\nsubseteq \Sset(\bar{\varphi})$. However, since level sets have the property $\Sset(\bar{\varphi}_1\subset \Sset(\bar{\varphi}_2)$ when $\bar{\varphi}_1)< \bar{\varphi}_2$, if we can show the claim for a large $\bar{\varphi}_2$, then the claim is also automatically true for all $\bar{\varphi}_1< \bar{\varphi}_2$.

Going back to the proof, by combining inequalities \eqref{eq:ineq-fciinf}, \eqref{eq:ineq-tri} with the definition of $\varphi$, we can derive the following lower bound:
\begin{multline}
  \tvarphi=\sum_{i=1}^N\tri f(\tci)\geq \sum_{i=1}^N (t-\rgi)f(\ciinf)\\
\geq (t-\rgmax)\sum_{i=1}^N f(\ciinf)\geq (t-\rgmax)\max_i f(\ciinf)\\
 \geq (t-\rgmax)\bar{f}\defeq \varphi_{lb},
\end{multline}
where $\bar{f}=\min_v\max_i f(v\transpose\betagi)$. Notice that
\begin{equation}
    \max_i f(v\transpose\betagi)=0 \iff f(v\transpose\betagi)=0 \;\forall i\iff v=\betagi\;\forall i.
\end{equation}
Hence, under the assumption of non-degenerate bearings, $\bar{f}$ is strictly positive, \ie, $\bar{f}>0$.

Now, for a given value $\bar{\varphi}$, we define the radius of the ball, $\bar{t}$, such that $\varphi_{lb}(\bar{t})=\bar{\varphi}$, \ie,
\begin{equation}
  \bar{t}=\frac{\bar{\varphi}}{\bar{f}}+\rgmax  
\end{equation}

We now show that $\Sset(\bar{\varphi})\subseteq B(x_g,\bar{t})$. From the proof of Proposition~\ref{prop:unique-critical-point}, we know that $\tvarphi$ is always monotonically increasing. Hence, there exist a unique $t^\ast$ such that $\tvarphi(t^\ast)=\bar{\varphi}$. This means that $\Sset(\bar{\varphi})$ is always star-shaped around $x_g$, and $t^\ast$ denotes the coordinate of the boundary for a given direction $v$. Since $\tvarphi(t^\ast)=\bar{\varphi}=\varphi_{lb}(\bar{t})$, and $\varphi(t^\ast)\geq\varphi_{lb}(t^\ast)$, it follows (by substituting the expression for $\varphi_{lb}$) that $t^\ast\leq \bar{t}$. It follows that, if we start from the origin $x_g$ along any line with direction $v$, we will never encounter the boundary of $\Sset(\bar{\varphi})$ after the boundary of $B(x_g,\bar{t})$. Hence, $\Sset(\bar{\varphi})\subseteq B(x_g,\bar{t})$ for any sufficiently large value of $\bar{\varphi}$. From the considerations above, this implies that $\varphi$ is radially unbounded and this completes the proof.

\subsection{Proof of Proposition~\ref{prop:functions}}
\label{sec:proof-functions}
We will compute the derivatives of each function and verify the requested properties.
  \begin{itemize}
  \item Proof for \eqref{eq:f-cosine} We have
    \begin{align}
      \df(c)&=-1,\\
      \ddf(c)&=0,\\
      f(c)+(1-c)\df(c)&=0.
    \end{align}
    Hence \eqref{eq:f-cosine} satisfies Assumption~\ref{ass:reshaping}.
  \item Proof for \eqref{eq:f-angleSq}: 
    \begin{align}
  \df(c)&=\begin{cases}-1 & \textrm{for } c=1\\-\frac{\arccos(c)}{\sqrt{1 - c^2}}& \textrm{for } -1\leq c <1\end{cases}\\
  \ddf(c)&=\frac{1}{1- c^2} - \frac{c\arccos(c)}{\sqrt{1 - c^2}^3}
\end{align}
Note that the value for $\df(c)$ at $c=1$ is fixed by left-continuity, as we will now show.  Recall that 
\begin{align}
  \frac{\de \arccos(c)}{\de c}&=-\frac{1}{\sqrt{1 - c^2}}\\
  \frac{\de \sqrt{1-c^2}}{\de c}&=-\frac{c}{\sqrt{1-c^2}}.
\end{align}
By applying L'H\^opital's rule, we have
\begin{equation}
  \lim_{c\to 1^-} -\frac{\arccos(c)}{\sqrt{1 - c^2}} = \lim_{c\to 1^-} -\frac{-\frac{1}{\sqrt{1 - c^2}}}{-\frac{c}{\sqrt{1-c^2}}}=\lim_{c\to 1^-} -\frac{1}{c} = -1
\end{equation}
 
Showing property \eqref{eq:property-f-derivativemix} for \eqref{eq:f-angleSq} requires some work.
\begin{multline}
f(c)+(1-c)\df(c) =\frac{1}{2}\arccos(c)^2 -\frac{(1-c)\arccos(c)}{\sqrt{1 - c^2}}\\
=\frac{1}{2}\arccos(c)\left(\arccos(c) -\frac{2(1-c)}{\sqrt{1 - c^2}}\right)
\end{multline}
Note that $\arccos(c)\geq 0$. Define
\begin{equation}
    b_s(c)=\arccos(c) -\frac{2(1-c)}{\sqrt{1 - c^2}}.
\end{equation}
Note that, strictly speaking, $b_s(c)$ is not defined for $c=1$. However, passing to the limit and recalling that $(1-c^2)=(1-c)(1+c)$, we can make the following simplification:
\begin{multline}
  \lim_{c\to 1^-} b_c(c)= \lim_{c\to 1^-} \arccos(c) -\frac{2(1-c)}{\sqrt{(1 - c)(1+c)}}\\=\lim_{c\to 1^-} \frac{2\sqrt{1-c}}{\sqrt{(1+c)}}=0
\end{multline}
Notice that this simplification would not be valid outside the limit, because it would be a division by zero. Taking the derivative of $b_s$, we have
\begin{equation}
  \frac{\de b_s(c)}{\de c}= \frac{(1-c)^2}{(1 - c^2)^\frac{3}{2}}\geq 0.
\end{equation}
This means that $b_s$ is a monotonically increasing function which approaches zero for $c=1$. Hence, $b_s$ and, consequently, $f(c)+(1-c)\df(c)$ are non-positive on $[-1, 1]$, as can it be also verified by simple visual inspection of a plot of the function. This shows the third required property for $f$.
  \end{itemize}
\subsection{Proof of Proposition~\ref{prop:Dbetagrad}}
\label{sec:proof-prop-Dbetagrad}
Similarly to what we did in \S\ref{sec:proof-gradient}, we use the definition~\eqref{eq:definition-differential}. However, this time, $x$ is fixed, and we define $\tbetagi=y_{gi0}+tv$ analogously to $\tx$.
\begin{multline}
\dert\gradx\varphi=-\df(\tci)\dtci\betai-\ddf(\tci)\dtci(I-\betai\betai)\tbetagi\\-\df(\tci)(I-\betai\betai\transpose)\dtbetagi\\
=-\df(\tci)\betai\betai\transpose\dtbetagi-\ddf(\tci)(I-\betai\betai)\tbetagi \betai\transpose\dtbetagi\\-\df(\tci)(I-\betai\betai\transpose)\dtbetagi\\
=-\bigl(\df(\tci)I+\ddf(\tci)(I-\betai\betai)\tbetagi\betai\bigr)\dtbetagi,
\end{multline}
from which~\eqref{eq:Dbetagrad} follows.

%%% Local Variables: 
%%% mode: latex
%%% TeX-master: "ICRA13-BearingControl.tex"
%%% End: 

}

\end{document}